\documentclass[12pt]{article}
\usepackage{geometry}                
\usepackage{graphicx}
\usepackage{url}
\usepackage{amssymb}
\usepackage{epstopdf}
\usepackage{amsthm}
\usepackage{amsmath}
\usepackage{hyperref}
\newtheorem{theorem}{Theorem}
\newtheorem{APthm}{Theorem}

\newtheorem{lemma}{Lemma}

\newtheorem{prop}{Proposition}

\newcommand{\R}{{\mathbb R}}
\newcommand{\Z}{{\mathbb Z}}
\newcommand{\N}{{\mathbb N}}

\newcommand{\ch}{{\bf 1}}

\newcommand{\ol}[1]{\overline{#1}}

\title{Principal Eigenvalue for Brownian Motion on a Bounded Interval  with Degenerate Instantaneous Jumps}  
\author{Iddo Ben-Ari\footnote{This work was partially supported by NSA grant H98230-12-1-0225, and a grant from the Simons Foundation (\#208728 to Iddo Ben-Ari)} \\~  Department of Mathematics\\ University of Connecticut\\ 196 Auditorium Rd \\ Storrs, CT 06269-3009\\~\\
\href{mailto:iddo.ben-ari@uconn.edu}{iddo.ben-ari@uconn.edu}}
\date{\today}
\begin{document}
\maketitle
\begin{abstract}
We consider a  model of Brownian motion on a bounded open interval with instantaneous jumps. The jumps occur at a spatially dependent rate given by a positive parameter times a continuous function positive on the interval and vanishing  on its boundary. At each jump event the process is redistributed uniformly in the interval.  We obtain sharp asymptotic bounds on the  principal eigenvalue for the generator of the process as the parameter tends to infinity. Our work answers a question posed by Arcusin and Pinsky. 
\end{abstract}
 
\section{Introduction and Statement of Results}

In a sequence of recent papers Pinsky \cite{P1} \cite{P2} and  Arcusin and  Pinsky  \cite{PA}  considered  the following model of a Brownian motion (elliptic diffusion in \cite{P2})  with instantaneous jumps.  Let $D\subset \R^d$ be a bounded domain, and let $\mu$ be a Borel probability measure on $D$ and $V \in C(\ol{D})$ a nonnegative function. Let $C_{\mu,V}u  :=V (x) \left ( \int u d\mu-u  \right),~u\in C_b (D)$, denote the generator of the pure-jump process  on $D$ with jump intensity $V$ and  a  jump (or more precisely, redistribution)  measure  $\mu$.  For $\gamma>0$, the diffusion with jumps process is generated by the non-local operator $-L_{\gamma,\mu,V}$, where
 \begin{equation} 
\label{eq:Lgamma} 
L_{\gamma,\mu,V} := - \frac 12 \Delta  -  \gamma C_{\mu,V}, 
 \end{equation}
  with the Dirichlet boundary condition on $\partial D$. In words, the process considered is Brownian motion killed when exiting $D$, and while in $D$, is redistributed at a spatially dependent rate $\gamma V$ according to measure $\mu$. The main object of study in the papers above was the asymptotic behavior of $\lambda_0(\gamma)$, the  principal eigenvalue for  $L_{\gamma,\mu,V}$,  as $\gamma\to\infty$.  The first paper, \cite{P1}, studies the model when  $V\equiv 1$. The second paper \cite{PA}  provides the nontrivial extension to the case where  $V$ is strictly positive on $\ol{D}$.  In what follows, we will refer to this positivity assumption as the ``nondegeneracy" condition. When $V$ is constant,  redistribution occurs at the jump times of Poisson of rate $\gamma V$, while for spatially dependent $V$ the jumps occur according to events of a a time-changed Poisson processes with constant rate $1$, time being sped up when $\gamma V$ is lager than $1$ and slowed down when $\gamma V <1$. The most recent paper \cite{P2} studies the model under the   nondegeneracy condition in the general setting of elliptic diffusions.  \\
  
  Let  $X:=(X(t):t\ge 0)$ denote  the process generated by $-L_{\gamma,\mu,V}$, and let $P_x^\gamma,E_x^\gamma$ denote the corresponding probability and expectation conditioned on  $X(0)=x\in D$. When  $\gamma=0$, we abbreviate and write $P_x$ and $E_x$. That is, $P_x$ and $E_x$ correspond to Brownian motion (no jumps).  Let 
   $$ \tau:=\inf \{ t>0: X(t) \not \in D\}$$
   denote the exit time of $X$ from $D$. Then $\lambda_0(\gamma)$ has the following probabilistic interpretation  \cite{PA}. For any $x\in D$, 
  \begin{equation} 
  \label{eq:spectral} 
   \lambda_0(\gamma) = - \lim_{t\to\infty} \frac {1}{t} \ln  P_x^\gamma (\tau>t).   \end{equation} 
Observe that  \eqref{eq:spectral} implies that given any $x\in D$, we have 
     \begin{equation} 
     \label{eq:momgen_rep}
      \lambda_0(\gamma) = \sup\{ \lambda\in \R : E_x^\gamma  \left (e^{\lambda \tau}\right) < \infty\}.
      \end{equation} 
In fact,  the limits and equalities in \eqref{eq:spectral} and \eqref{eq:momgen_rep} remain to hold when replacing the probability $P_x^\gamma$ and expectation $E_x^\gamma$ with $\sup_{x \in D}  P_x^\gamma$ and $\sup_{x\in D} E_x^\gamma$, respectively. \\

The above cited papers provide sharp asymptotic behavior for $\lambda_0(\gamma)$ as $\gamma\to \infty$, under the nondegeneracy condition and smoothness assumptions on $\partial D$ and $\mu$. In particular, the following result was obtained. 
 
    \begin{APthm}[\cite{PA}, Theorem 1-i]
    \label{th:nondeg} 
     Assume that 
    $D$ has  $C^{2,\beta}$-boundary  for some $\beta \in (0,1)$,  $\min_{x \in \ol{D}} V(x)>0$, and for some $\epsilon>0$, $\mu$ possesses a density in $C^1(\ol{D}^\epsilon)$, where $D^\epsilon:=\{x\in D: d(x,\partial D)<\epsilon\}$, then 
  \begin{equation}
  \label{eq:limlim}  \lim_{\gamma \to\infty} \frac{ \lambda_0(\gamma)}{\sqrt{\gamma}}= \frac{ \int_{\partial D} \frac{\mu}{\sqrt{V}} d\sigma}{\sqrt{2} \int_D \frac{1}{V} d\mu},
  \end{equation}
   where $\sigma$ is the Lebesgue measure on $\partial D$. 
   \end{APthm}
   We comment that \cite[Theorem 1]{PA} includes an additional statement generalizing the result to  $\mu$ with  smooth density near $\partial D$, vanishing up to the $\ell$-th order for some $\ell\in \Z_+$.\\
      
   The nondegeneracy condition  could be viewed as one  extreme, the other extreme being the case where $V$ is compactly supported. It was noted in \cite{PA} that when the support $K$ of  $V$ is compact,  then for $x\in D\backslash K$, and for any $\gamma>0$,  the distribution of $\tau$ under $P_x^{\gamma}$  dominates the exit time for the Brownian Motion (no jumps) from $D\backslash K$, and hence it follows from \eqref{eq:spectral} that  $\lambda_0$ is bounded above by the principal eigenvalue for $-\frac 12 \Delta$ on $D\backslash K$, a positive constant independent of $\gamma$. \\
   
    In light of the above,  is it reasonable to expect some transition in the behavior of $\lambda_0$ from the nondegenerate case to the compactly supported case  to occur when $V$ is positive on $D$ and vanishes on $\partial D$.  The behavior in this regime was left as an open problem in \cite{PA}. In this paper we answer it in one dimension. Our method is based on analysis of the moment generating function in  \eqref{eq:momgen_rep},  obtained through  probabilistic arguments. \\
    
         In what follows, for real-valued functions $f,g$ with domain $D$, and   $a \in \partial D$ or $a$ taken as $\partial D$, we write   $f(x)\underset{x\to a} {\asymp} g(x)$   meaning  $0< \liminf_{x\to a} f(x) / g(x) \le \limsup_{x\to a} f(x) / g(x)<\infty$, whenever the limits make sense.  This notation will be also used when  $f,g$ are real-valued functions on $(0,\infty)$,  and $a$ taken as $0$ or $\infty$. \\
         
Before stating our main result, we present some heuristics derived from Theorem \ref{th:nondeg}, which provide some indication on the behavior  when $V$ vanishes on $\partial D$. Assume that  $\mu$ is uniform on $D$ and that $V(x) \underset{x\to \partial D}{\asymp} d(x,\partial D)^{\alpha}$ for some $\alpha>0$. Observe that \eqref{eq:limlim} is not well-defined also because  the surface integral in the numerator of the right-hand side blows up. We can approximate it  through volume integrals of the form 
   $$  \frac{\int_{D^\epsilon} \frac{ d\mu }{\sqrt{V}}}{\int_{D^\epsilon} d \mu}\underset{\epsilon\to 0} {\asymp} \epsilon^{-\alpha/2},~\alpha \ne 2,$$
   where $D^\epsilon$ is as in Theorem \ref{th:nondeg} (note that the ratio approximates the integral with respect to the normalized Lebesgue measure, therefore a  positive multiplicative constant is missing. Since this constant  has no effect on the argument, we will ignore it).   When $\alpha<1$, the volume integral in the denominator of \eqref{eq:limlim} converges, therefore  letting $\epsilon \to 0$ in the approximation above,  the ratio blows up,  giving  the prediction $\sqrt{\gamma} = o(\lambda_0(\gamma))$. When $\alpha\ge 1$, the denominator also blows up, suggesting a possible  phase transition at  $\alpha=1$.  For $\alpha>1$, we can approximate the volume integral in the denominator  by integrating over $D-D^\epsilon$ instead of $D$. Then, 
    $$ \int_{D-D^\epsilon} \frac{d\mu }{V} \underset{\epsilon\to 0} {\asymp} \epsilon^{1-\alpha}.$$  
Combining both approximations  (with same $\epsilon$, this is not a rigorous treatment),  we obtain an approximation to the ratio, proportional to  $\epsilon^{-\frac{\alpha}{2}}/ \epsilon ^{1-\alpha} = \epsilon^{\frac{\alpha}{2}-1}$, as $\epsilon \to 0$.  This blows up as $\epsilon \to 0$ when $\alpha \in (1,2)$, converges to $1$ when $\alpha=2$ and converges to $0$ when $\alpha>2$. Summarizing, the heuristics suggest that $\sqrt{\gamma} = o(\lambda_0(\gamma))$ for $\alpha \in (1,2)$,  while $\lambda_0(\gamma) \asymp \sqrt{\gamma}$ for $\alpha=0,2$, and  $\lambda_0(\gamma) = o(\sqrt{\gamma})$ for $\alpha>2$. \\

Here is our main result. 

 \begin{theorem}
 \label{th:asymp} 
Let $D=(0,1)$ and $\mu$ denote the Lebesgue measure on $D$. Assume that $V\in C(\ol{D})$ satisfies     $V>0$ on $D$, and for some $0\le \alpha'\le \alpha <\infty$, $V(x)\underset{x\to 0^+}{ \asymp} x^{\alpha}$, and $V(x) \underset{x\to 1^-}{\asymp} (1-x)^{\alpha'}$. Let $\delta =\delta (\alpha) = \frac{\alpha\wedge 1+1}{\alpha+2}$. Then 
  \begin{equation}
  \label{eq:asymp}  \lambda_0(\gamma) \underset{\gamma\to \infty} {\asymp} \gamma^{\delta(\alpha)} \times  \begin{cases} \frac{1}{\ln \gamma} & \alpha=1; \\ 1 & \mbox{otherwise}.\end{cases} 
  \end{equation} 
 \end{theorem}

 

We would like to note the following.

\begin{enumerate} 
\item Observe that $\delta(\alpha')$ may be larger or smaller than $\delta(\alpha)$, yet the asymptotic behavior is determined by the larger parameter $\alpha$. This is a result of  the fact that in the formula for the moment generating function for $\tau$, expressed in terms of the Brownian motion,  the function  $V$ appears as a penalizing potential, discounting paths which spend more time at sets where $V$ is larger.
\item  The nondegeneracy condition is covered by the case $\alpha=0$. 
\item The graph of $\delta$ is shown in Figure \ref{fig:delta}. Note the phase transition at $\alpha=1$. The Theorem corroborates the heuristic derivation preceding it. 
\end{enumerate} 

\begin{figure}
    \includegraphics[scale=0.3]{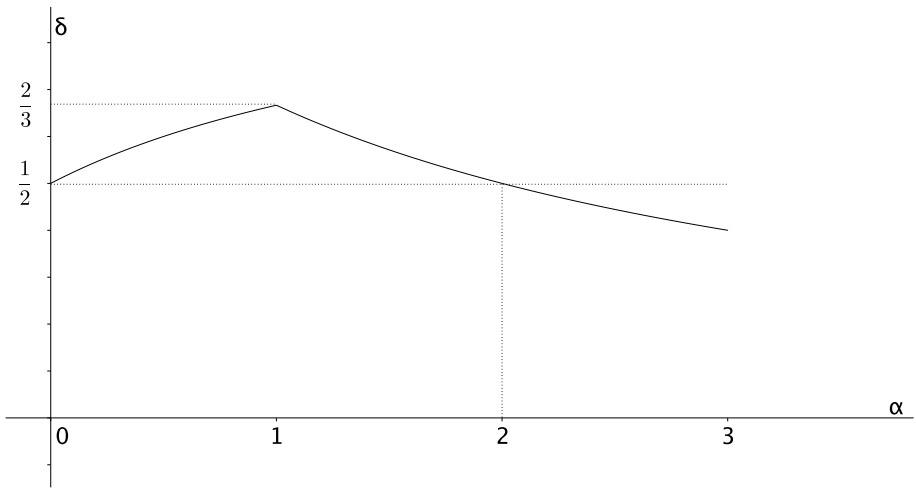}
    \caption{Graph of $\delta$}
    \label{fig:delta}
\end{figure}

The remainder of the paper is organized as follows. In Section \ref{sec:momgen} we prove some identities and a lower bound  on the moment generating function of $\tau$. In Section \ref{sec:BM_comp} we obtain the main estimates on functions of Brownian motion, which when  combined with the results of Section  \ref{sec:momgen} yield the proof of Theorem \ref{th:asymp}.  This proof is given in Section \ref{sec:proof}. 
  
 \section*{Acknowledgement} 
 The author would like to thank Ross Pinsky for helpful discussions and useful suggestions, and to an anonymous referee for carefully reading the manuscript and helping improve the presentation. 
 
\section{The Moment Generating Function}
We define a family of stopping times for $X$. For $y \in D$, we let 
 \begin{equation} 
 \label{eq:tau_y} 
 \tau_{y}:= \inf \{t \ge 0: X(t) =y\}.
 \end{equation} 
\label{sec:momgen}
We begin by recalling a well-known classical result about the moment generating function of the exit time of Brownian motion from an interval (see e.g. \cite[pp. 71-73]{RY}).
\begin{prop}
\label{pr:momgen} 
 Let $0\le  a<y<b\le 1$ and let $\rho>0$. For $i=a,b$, let $A_i:=\{ \tau_a \wedge \tau_b = \tau_i\}$, and let $j:=a$ if $i=b$ and $j:=b$ otherwise.  Then we have : 
 \begin{enumerate} 
 \item  
    \begin{equation*} 
    E_y \left ( e^{-\rho \tau_i  }\ch_{A_i } \right) = \frac{ \sinh (\sqrt{2\rho} |y-j| ) }{\sinh(\sqrt{2\rho}(b-a) )},\mbox{ and } E_y \left ( e^{-\rho\left (  \tau_a \wedge \tau_b\right)} \right) = \frac { \cosh (\sqrt{2\rho} (y- \frac{a+b}{2}))}{\cosh (\sqrt{2\rho}\frac {b-a}{2})}.
    \end{equation*} 
\item If $\sqrt{2\rho }(b-a) < \pi $, then 
    \begin{equation*} 
   E_y  \left( e^{\rho \tau_i} \ch_{A_i}\right) =  \frac{ \sin( \sqrt{2\rho} |y-j| )}{\sin(\sqrt{2\rho } (b-a))}
     \end{equation*}
     If $\sqrt{2\rho} (b-a) \ge \pi$, then the expectation above is infinite. 
     \end{enumerate} 
\end{prop} 
  \begin{prop}
  \label{pr:crudebd} 
  There exists a  constant $\theta_0\in (0,\infty)$ depending only on $V$   such that if $\lambda \ge \theta_0  \gamma^{\frac{2}{\alpha+2}}$, then 
  $$ E_\mu^\gamma ( e^{\lambda  \tau}) = \infty.$$ 
 \end{prop} 
 \begin{proof}
Fix $x \in (0,\frac 14)$. Let $\sigma_x := \tau \wedge \tau_{2x}$  denote the exit time of the diffusion from the interval $(0,2x)$. Under $P_x^\gamma$,  $\tau \ge \sigma_x \wedge J$, where $J$ is the time of the first jump.  Since the jump rate on the interval $(0,2x)$ is bounded above by $\rho:= c_1 \gamma x^\alpha$,    it follows that $\tau$  stochastically dominates $\sigma_x \wedge J'$ where $J'$ is exponential with rate $\rho$, independent of $\sigma_x$.  Let $\lambda > \rho$. Conditioning on $J'_x$, we obtain 
\begin{align*} 
  E^\gamma_x \left ( e^{\lambda \tau}\right)  &\ge E^\gamma_x\left (  e^{\lambda (\sigma_x \wedge J'_x) } \right) = E_x \left ( \rho \int_0^{\infty} e^{\lambda   (\sigma_x\wedge y)  } e^{-\rho y} dy \right) \\
  & = E_x \left ( \rho \int_0^{\sigma_x} e^{(\lambda-\rho)y} dy + e^{(\lambda-\rho) \sigma_x} \right ) \\
  & = \frac{ \rho}{\lambda-\rho } \left (  E_x e^{(\lambda-\rho) \sigma_x}-1 \right)  + E_x\left (  e^{(\lambda-\rho) \sigma_x}\right).\\
  &=  \frac{\lambda}{\lambda-\rho} E_x \left ( e^{(\lambda-\rho) \sigma_x}\right)  - \frac{\rho}{\lambda-\rho}.
  \end{align*}
  From Proposition \ref{pr:momgen}-(2) we conclude that  
 $E_x \left ( e^{(\lambda-\rho) \sigma_x}\right)<\infty$ if and only if 
$\lambda -\rho < \frac{\pi^2}{8x^2}$. Thus,  whenever $\lambda \ge  \rho + c_2 x^{-2}=c_1 \gamma x^\alpha + c_2 x^{-2}$,  one has $E^\gamma_x \left ( e^{\lambda \tau}\right)  =\infty$. 
Suppose now that $x= c\gamma^{-\frac{1}{\alpha+2}}$ for some $c>0$.  Then $E_x^\gamma \left ( e^{\lambda \tau}\right)=\infty$, provided that 
 $$ \lambda \ge c_1 \gamma c^\alpha  \gamma ^{-\frac{\alpha}{\alpha+2}} +c_2 c^{-2} \gamma^{\frac{2}{\alpha+2}} = \left ( c_1 c^\alpha + c_2 c^{-2} \right)  \gamma^{\frac{2}{\alpha+2}}.$$
 Let $\theta_0: = \min_{c \in [1,2]} \left (c_1 c^{\alpha} + c_2 c^{-2} \right)$. If $\lambda \ge  \theta_0 \gamma^{\frac{2}{\alpha+2}}$, then 
   $$ E^\gamma_\mu \left ( e^{\lambda \tau}\right)  \ge \int_{x\in [1,2] \gamma^{-\frac{1}{\alpha+2}}} E_x^\gamma \left (e^{\lambda \tau} \right)  dx  = \infty.$$
\end{proof} 
For $\lambda \in \R$ and $t\ge 0$, let
 $$ R_\lambda(t):= \lambda t - \gamma \int_0^t V(X(s)) ds.$$  
 We have the following proposition,  expressing the moment generating function purely in terms of Brownian expectations. 
  \begin{prop} Let $x\in D$. Then 
  \label{pr:Rmomgen} 
  \begin{enumerate} 
  \item   $\displaystyle E^\gamma_x  \left ( e ^{\lambda \tau}\right) = E^\gamma_x\left  ( e^{\lambda J} \ch_{\{J< \tau\}}\right)  E^\gamma_\mu \left(e^{\lambda \tau}\right) + E^\gamma_x \left( e^{\lambda \tau}\ch_{\{\tau < J\}}\right).$
  \item $\displaystyle E^\gamma_x \left (e^{\lambda \tau} \ch_{\{\tau<J\}}\right) = E_x \left( e^{R_\lambda (\tau)}\right).$
  \item $\displaystyle E^\gamma_x \left (e^{\lambda J} \ch_{\{J<\tau\}}\right)=  \lambda E_x \left( \int_0^\tau e^{R_\lambda (t)} \right) + 1-  E_x \left(e^{R_\lambda (\tau)}\right).$
  \item $\displaystyle  E^\gamma_\mu (e^{\lambda \tau} )= \frac{1}{1- \frac{ \lambda E_\mu \left ( \int_0^\tau e^{R_\lambda (t) } dt\right)}  {E_\mu \left(e^{R_\lambda(\tau)}\right)}}.$
  \end{enumerate}
  \end{prop}
  This result essentially allows to reduce the problem to estimating the asymptotic behavior of the Brownian expectations appearing on the right-hand side of each of the identities. This is carried out in Section \ref{sec:BM_comp} below. Since  these expectations are also solutions to some related ordinary differential equations, it is interesting to ask for independent analysis not based on the probabilistic analysis. Specifically, let  ${\cal A}$ denote the differential operator ${\cal A}u := \frac 12 u'' + ( \lambda - \gamma V) u $. Then  $E_x \left ( e^{R_\lambda (\tau)} \right ) $ is  known as the gauge associated to ${\cal A}$ on $D$, that is, the solution to 
 $$ \begin{cases} {\cal A} u = 0 & \mbox{ on D}\\ u|_{\partial D} =1,\end{cases} $$
  and $E_x  \left( \int_0^\tau e^{R_\lambda (t)} dt \right)$ is a potential for ${\cal A}$ on $D$, or total mass of Green's measure, solving : 
  $$ \begin{cases} {\cal A} u = -1& \mbox{ on D}\\ u|_{\partial D}=0.\end{cases}$$
  \begin{proof} 
   The first identity  follows directly  from the strong Markov property. Integrating both sides of the first identity with respect to $\mu$ we obtain 
\begin{equation} 
\label{eq:int_id} 
 E^\gamma_\mu \left( e ^{\lambda \tau}\right) = E^\gamma_\mu  \left( e^{\lambda J} \ch_{\{J< \tau\}}\right)   E^\gamma_\mu \left(e^{\lambda \tau} \right) + E^\gamma_\mu \left( e^{\lambda \tau}\ch_{\{\tau < J\}}\right).
 \end{equation} 

  In what follows we assume $\lambda $ is less than the principal eigenvalue for $-\frac 12 \Delta$ on $D$. In particular, $\sup_x E_x \left ( e^{\lambda \tau}\right) <\infty$. The identities (2)-(4) extend beyond this domain by analyticity. To prove the second identity, observe that 
$$  e^{\lambda \tau}\ch_{\{\tau < J\}}= \left ( \lambda \int_0^\infty e^{\lambda t} \ch_{\{\tau>t\}} dt +1 \right) \ch_{\{\tau< J\} }.$$
Write $I(t):= \int_0^t V(X(s)) ds$. Thus, 
\begin{align}
\nonumber
 E^\gamma_x \left(e^{\lambda \tau}\ch_{\{\tau < J\}}  \right) &= \lambda \int_0^\infty e^{\lambda t} P_x^\gamma ( \tau>t; \tau<J) dt + P_x^\gamma(\tau<J)\\
 \nonumber
 & = \lambda E_x \left ( \int_0^\infty e^{\lambda t} \ch_{\{ \tau>t\}} e^{-\gamma I(\tau) }dt \right) + E_x \left(e^{-\gamma I(\tau)}\right)\\
 \nonumber
 & = E_x \left ( (e^{\lambda \tau}-1) e^{-\gamma I(\tau)}\right ) +E_x \left(e^{-\gamma I(\tau)}\right) \\
 \label{eq:taubeforeJ}
 & = E_x \left(e^{R_\lambda(\tau)}\right).
 \end{align}
 This proves the second identity. 
 We turn to the third identity. 
\begin{align*} 
   e^{\lambda J} \ch_{\{J < \tau\}} &= \left ( \lambda  \int_0^\infty e^{\lambda t} \ch_{\{J >t\}} +1 \right ) \ch_{\{J<\tau\}}\\
    & = \lambda \int_0^\infty e^{\lambda t} \ch_{\{\tau>t\}} ( \ch_{\{J >t\}} - \ch_{\{\tau<J\}}) dt +\ch_{\{J<\tau\}}.
\end{align*}
Thus, 
\begin{align}
\nonumber 
 E^\gamma_x 
\left (e^{\lambda J}\ch_{\{J<\tau \}}  \right) &= \lambda \int_0^\infty e^{\lambda t} P_x^\gamma (\tau \wedge J >t ) dt - \lambda \int_0^\infty e^{\lambda t} P_x^\gamma (\tau>t; \tau<J)dt +  P_x^\gamma(J<\tau )\\
\nonumber
 & \overset{\eqref{eq:taubeforeJ}}{=} \lambda E_x \left ( \int_0^\tau  e^{R_\lambda(t) }dt \right)-\left ( E^\gamma_x \left ( e^{\lambda \tau}\ch_{\{\tau<J\}}\right) -P_x^\gamma (\tau<J)\right )  + 1-P_x^\gamma (\tau<J)\\
  \label{eq:whatheck}
& \overset{\eqref{eq:taubeforeJ}}{=} \lambda E_x \left( \int_0^\tau e^{R_\lambda(t) } dt \right) + 1-  E_x \left(e^{R_\lambda(\tau)} \right).
 \end{align}
It remains to prove the last identity. Observe that $\lambda \int_0^\tau e^{R_\lambda (t)} dt +1 \le e^{\lambda \tau}$, and that $e^{R_\lambda (\tau)}\le e^{\lambda \tau}$. Therefore since $\sup_x E_x \left ( e^{\lambda \tau}\right) < \infty$ by assumption, it follows from dominated convergence applied to the right-hand side of \eqref{eq:whatheck} that  
$$\lim_{\lambda \to 0} E_{\mu}^\gamma \left ( e^{\lambda J} \ch_{\{J<\tau\}}\right) =1- E_\mu \left ( e^{-\gamma \int_0^{\tau} V(X(t)) dt}\right) = P_\mu^\gamma(J<\tau)<1.$$ 
  Consequently, we obtain from  \eqref{eq:int_id} that 
 $$ E_{\mu}^\gamma  \left ( e^{\lambda \tau}\right) = \frac{ E_{\mu}^{\gamma} \left ( e^{\lambda \tau} \ch_{\{\tau < J\}}\right)}{1-  E_{\mu}^\gamma \left ( e^{\lambda J} \ch_{\{J<\tau\}}\right)},$$
 and the right-hand side is finite. 
Plugging the second and third identities into this we obtain 
  $$ E_{\mu}^\gamma  \left ( e^{\lambda \tau}\right) = \frac{ E_\mu \left ( e^{R_{\lambda} (\tau)} \right) }{E_\mu \left ( e^{R_{\lambda} (\tau)} \right) -\lambda E_{\mu} \left ( \int_0^\tau e^{R_\lambda (t)} dt \right)  },$$
   and the result follows.
 \end{proof}
\section{Brownian Computations} 
\label{sec:BM_comp}
In this section we obtain the main estimates needed to prove Theorem \ref{th:asymp}. We need some definitions. 
Below we let $r=r(\gamma)= r(\gamma,\alpha) := \gamma^{-\frac {1}{\alpha+2}}$, and 
 $$ h=h(\gamma)=h(\gamma,\alpha) := \begin{cases} r(\gamma) & \alpha<1; \\ r(\gamma)/\ln \gamma & \alpha=1; \\ r(\gamma)^{\alpha} & \alpha>1.\end{cases}$$
  The function $h$ was chosen to satisfy that $\gamma h(\gamma)$ is equal to the right-hand side of \eqref{eq:asymp}. We also define a function  $\lambda=\lambda(\theta,\gamma,\alpha)$  by letting
 \begin{equation}
 \label{eq:lambda} \lambda (\theta,\gamma,\alpha) :=\theta \gamma h(\gamma)=\theta \begin{cases} \gamma^{\frac{\alpha+1}{\alpha+2}}& \alpha <1;\\
 \frac{ \gamma^{\frac 23}}{\ln \gamma}& \alpha=1;\\ \gamma^{\frac{2}{\alpha+2}}& \alpha >1.\end{cases}
 \end{equation} 
 In what follows, in order to simplify notation, we sometimes omit the dependence of the functions $r,h$ and $\lambda$ on some of their arguments. \\
 
 We begin with  following simple lemma needed for  our estimates and whose proof will be omitted. 

 \begin{lemma} ~
 \label{lem:hasit} 
 \begin{enumerate} 
 \item For $\gamma\ge e$,     $h(\gamma) \le  r(\gamma)^{\alpha}$, and when $\alpha \le 1$, one has $h(\gamma) = o(r^\alpha (\gamma))$ as $\gamma \to \infty$. 
 \item  For $c >0$, $h(\gamma) \int_{r(\gamma)<x <  c} \frac{1}{x^\alpha} dx \underset{\gamma \to\infty}{\asymp} r(\gamma)$
 \end{enumerate} 
 \end{lemma}
 \begin{lemma}
 \label{lem:upperbd}
 There exists a constant $\theta_1 \in (0,\infty]$ and positive constants $C_1,C_2$ depending only on $V$, such that if  $\theta < \theta_1$  then there exists a positive constant   $\gamma_1:=\gamma_1(\alpha,\theta)$ and  $\gamma > \gamma_1$ implies 
 $$ E_\mu \left (  e^{R_\lambda (\tau)}\right)  \le C_1  r(\gamma),$$
  and 
  $$ |\lambda|  E_\mu \left ( \int_0^{\tau} e^{R_\lambda(t)} dt \right)\le C_2 |\theta|  r(\gamma),$$
Furthermore
  \begin{enumerate} 
  \item 
  For fixed $\alpha$, the function $\theta \to \gamma_1(\alpha,\theta)$ is nondecreasing. 
  \item If $\alpha\le 1$ then   $\theta_1 = \infty$.
  \end{enumerate} 
 
   \end{lemma} 

\begin{proof} 
We first need some preparation  before getting into the main argument. 
The preparation consists of several steps. The first is a reduction to symmetric $V$.  
     Since $\alpha >\alpha'\ge 0$, we have that $V(x) \underset{x\to 0^+} {\asymp}x^\alpha \le x^{\alpha'} \underset{x\to 0^+}{\asymp} V(1-x)$. Since in addition $V$ is strictly positive and continuous on $D$,  we can find $\hat V\in C(\ol{D})$ such that $V> \hat V>0$ in $D$, $\hat V(x)\underset{x\to\partial D}{\asymp} d(x,\partial D)^\alpha$,  and $\hat V$ is symmetric. That is $\hat V(1-x) = \hat V(x)$. Letting $\hat R_\lambda$ denote the analog of $R_\lambda$ with $\hat V$ in place of $V$. Then  $R_\lambda \le \hat R_\lambda$. Therefore to prove the lemma, there is no loss of generality assuming that  $V$ is symmetric and $V(x) \underset{x\to\partial D}{\asymp} d(x,\partial D)^\alpha$. \\
     
     The next step in the preparation is to obtain the constants $\theta_1,\gamma_1$ in the Lemma. We need to define a family of stopping times for the Brownian motion. For $l \in (0,\frac 12]$, let 
     $$ \sigma_l:= \inf  \{t \ge 0 : d(X_t,\partial D) = l\}.$$
     Therefore $\sigma_l = \tau_l \wedge \tau_{1-l}$, where $\tau_l$ was defined in \eqref{eq:tau_y}. Let $\delta>0$ be such that  $V(x) \ge \delta d(x,\partial D)^{\alpha}$ for all $x \in D$. Choose $\kappa>1$ such that $\delta \kappa^\alpha\ge 2$, and let $r_j(\gamma) = \kappa^j r(\gamma)$ for $j=1,2,3$. Below we will omit the dependence of $r_j$ on $\gamma$. When $\alpha>1$, let $\theta_1:= \frac{\pi^2}{32 \kappa^6}<1$, and otherwise let $\theta_1:=\infty$. Assume that $\theta < \theta_1$. Since we are looking for upper bound, there is no loss of generality assuming  $\theta >0$. 
 Assume first that  $\alpha \le 1$. Since by  Lemma \ref{lem:hasit}-(1), $h=o(r^\alpha)$, we can find $\gamma_1:=\gamma_1(\alpha,\theta)<\infty$ such that for all $\gamma> \gamma_1$,  $h/r^\alpha$  satisfies  $h/r^\alpha < \frac{ \pi^2}{32 \kappa^6 \theta}<\frac 1\theta$. In addition, for fixed $\alpha$, the function $\theta \to \gamma_1(\alpha,\theta)$ could be chosen  as nondecreasing.  We have  
  $ \lambda = \theta \gamma h < \gamma r^\alpha$, as well as 
  $$\sqrt{2\lambda} r_3=  \sqrt{2\theta} \gamma^{1/2} h^{1/2} \kappa^3 r = \sqrt{2\theta }\kappa^3  r^{-\frac{\alpha+2}{2}+1} h^{1/2} = \sqrt{2\theta } \kappa^3 (h/r^{\alpha})^{1/2} < \frac{\pi}{4}.$$
    If  $\alpha> 1$, then $h=r^{\alpha}$ and since $\theta < \theta_1$, and $\theta_1<1$, we obtain 
    $$\lambda < \gamma r^{\alpha}\mbox{ and }  \sqrt{2\lambda} r_3 < \sqrt{2\theta_1} \gamma r^{\alpha/2+ 1} \kappa^3 =\frac{\pi}{4}$$
     for all $\gamma>0$. In this case we set $\gamma_1(\alpha,\theta):=0$. 
    Summarizing both cases, we proved that there exists $\theta_1 \in (0,\infty]$ and  $\gamma_1(\alpha,\theta)$, nondecreasing in $\theta$   such that for $\theta <\theta_1$ and $\gamma\ge \gamma_1$ we have  
     \begin{equation} 
     \label{eq:lambdasmll}
      \lambda < \gamma r^\alpha\mbox{ and } \sqrt{2\lambda} r_3 < \frac{\pi}{4},
     \end{equation} 
For the remainder of the proof we assume that $\theta \in (0,\theta_1)$ and $\gamma \ge \gamma_1$. \\

The next and the final  step in the preparation consists of several estimates to be later used. Let   $\rho := \lambda$, $a=0$ and $b:=r_3$. Then  $\sqrt{2\rho}(b-a) = \sqrt{2\lambda} r_3 < \frac{\pi}{4}$, and  it follows from Proposition \ref{pr:momgen}-(2) that  for $0<y<r_3$       $$E_y  \left(   e^{\lambda \tau}\ch_{\{\tau<\sigma_{r_3}\}}\right) =  \frac{ \sin( \sqrt{2\lambda } (r_3-y))}{\sin(\sqrt{2\lambda} r_3)}\mbox{, and }  E_y \left( e^{\lambda \sigma_{r_3}  }\ch_{\{\sigma_{r_3} <\tau\}}\right)  =  \frac{ \sin( \sqrt{2\lambda } y)}{\sin(\sqrt{2\lambda} r_3)}.$$
        Since $t\to \sin(t)$ is increasing on $[0,\frac{\pi}{4}]$, we obtain 
       \begin{equation}
       E_x  \left(   e^{\lambda \tau}\ch_{\{\tau<\sigma_{r_3}\}}\right) \le 1\mbox{, and }E_x \left( e^{\lambda  \sigma_{r_3}  }\ch_{\{\sigma_{r_3} <\tau\}}\right)  \le   \frac{ \sin( \sqrt{2\lambda } r_2)}{\sin(\sqrt{2\lambda} r_3)}\le \frac {c_1}{1+c_1}<1,
        \label{eq:allbest}
       \end{equation} 
        where $c_1$ is the universal constant  satisfying $ \frac{c_1}{1+c_1}  = \sup_{t\in (0,\frac \pi 4)}\frac{\sin({\kappa^{-1}  t})}{\sin(t)}\in (0,1)$. \\
        
 Suppose that  $x\in D$ satisfies  $d(x,\partial D) \ge r_1$, and  assume $ 0\le s \le t \le \sigma_{r_1}$. Clearly,  
     $$V(X(s)) \ge \delta d(X(s),\partial D)^\alpha\ge \delta r_1^{\alpha}=\delta (\kappa r)^{\alpha} \ge 2 r^{\alpha}.$$
      Combining this with the first inequality in \eqref{eq:lambdasmll}, we obtain  $\lambda < \frac \gamma2  V(X(s))$.  Summarizing, 
    \begin{equation} 
   \label{eq:mainuppr} 
     R_\lambda(t) \le - \frac \gamma2 \int_0^t V(X(s)) ds\le -  \gamma r^{\alpha} t,~P_x\mbox{ a.s.},
     \end{equation} 
     when $d(x,\partial D) \ge r_1$ and $t \in [0,\sigma_{r_1}]$. We now obtain a similar upper bound in terms of $x$.  Suppose $d(x,\partial D) \ge r_2$. Without loss of generality, let  $x \in[r_2,\frac 12]$. Next, if $y\in D$ is such that $d (y,\partial D) \ge  \kappa^{-1} x$, then   $V(y) \ge \delta d(y,\partial D)^\alpha\ge \delta (\kappa^{-1} x)^{\alpha}$. As a result, if $t \in [0,\sigma_{\kappa^{-1}x}]$, we have 
$$ R_\lambda (t ) \le \lambda t - \delta\gamma   (\kappa^{-1} x)^{\alpha}t ,~P_x\mbox{ a.s.}$$ 
But by \eqref{eq:lambdasmll} and the fact that $\delta \kappa^{\alpha}\ge 2$, we have 
 $$\lambda < \gamma r^{\alpha} = \gamma (r_2 \kappa^{-2})^{\alpha} < \gamma \kappa^{-\alpha} ( x \kappa^{-1} )^{\alpha}  < \gamma\frac{\delta}{2} (x \kappa^{-1})^{\alpha}.$$
  Therefore, letting $c_2:= \delta \kappa^{-\alpha}/2$, we obtain 
   \begin{equation} 
   \label{eq:Rlambdaxfar} 
   R_\lambda(t) \le -c_2 \gamma x^{\alpha} t,~P_x\mbox{ a.s.},
   \end{equation} 
   provided $d(x,\partial D) \ge r_2$ and $t \in [0,\sigma_{\kappa^{-1} x}]$. \\
     
     We turn to the main proof, beginning with the first bound.  Fix $K\in\N$ and let $x\in \partial D$ satisfy $d(x,\partial D)\le r_2$. 
   By the Strong Markov property, 
      \begin{equation*} 
   \label{eq:comp} E_x \left ( e^{R_{\lambda} (\tau)\wedge K} \right ) \le E_x \left ( e^{\lambda \tau }\ch_{\{\tau<\sigma_{r_3}\}} \right)+ 
  E_x  \left (   e^{\lambda \sigma_{r_3} }\ch_{\{\sigma_{r_3}<\tau\}}\right) E_{r_3}\left( e^{R_{\lambda} (\tau)\wedge K}\right),
  \end{equation*} 
  and 
  \begin{equation*}
\label{eq:mormark}
 E_{r_3} \left (e^{R_{\lambda} (\tau)\wedge K} \right)  \le E_{r_3} \left ( e^{R_{\lambda} (\sigma_{r_1})}\right) E_{r_1} \left( e^{R_{\lambda}(\tau)\wedge K}\right).
 \end{equation*} 
 It follows from \eqref{eq:mainuppr} that $E_{r_3} \left ( e^{R_{\lambda} (\sigma_{r_1})}\right)<1$. Therefore 
  \begin{align}
     \nonumber
   E_x \left ( e^{R_{\lambda} (\tau)\wedge K} \right ) &\le  E_x \left ( e^{\lambda \tau }\ch_{\{\tau<\sigma_{r_3}\}} \right)+ 
  E_x  \left (   e^{\lambda \sigma_{r_3} }\ch_{\{\sigma_{r_3}<\tau\}}\right) E_{r_1} \left( e^{R_{\lambda}(\tau)\wedge K}\right)\\
   \label{eq:smll_rep2} 
  &\overset{\eqref{eq:allbest}}{ \le} 1 + \frac{c_1}{1+c_1} E_{r_1} \left( e^{R_{\lambda}(\tau)\wedge K}\right).
  \end{align}
Letting $x=r_1$, we obtain $  E_{r_1} \left( e^{R_{\lambda}(\tau)\wedge K}\right)\le 1+c_1$, and plugging the latter inequality back into \eqref{eq:smll_rep2}, we obtain  $E_x \left ( e^{R_{\lambda} (\tau)\wedge K} \right ) \le 1+ c_1$. Finally, letting  $K\to\infty$, monotone convergence gives
\begin{equation} 
    \label{eq:max_smll}
     E_x \left ( e^{R_{\lambda}(\tau)}\right) \le 1+c_1,   \end{equation} 
    when  $d(x,\partial D) \le r_2$.\\
   
Next we find an upper bound on $E_x \left(e^{R_\lambda(\tau)}\right)$ when $d(x,\partial D)\ge r_2$. Assume then that $x \in [r_2,\frac 12]$.  By the Strong Markov property,     
\begin{align}
\nonumber
      E_x \left ( e^{R_{\lambda} (\tau)} \right)& = E_x \left ( e^{R_{\lambda} (\sigma_{\kappa^{-1} x} )}\right) E_{\kappa^{-1}   x}\left (  e^{R_{\lambda} (\sigma_{r_1})}\right) E_{r_1} \left ( e^{R_{\lambda} (\tau)} \right)\\
      \nonumber 
      & \overset{\eqref{eq:mainuppr}}{\le}  E_x \left ( e^{R_{\lambda} (\sigma_{\kappa^{-1} x} )}\right)E_{r_1} \left ( e^{R_{\lambda} (\tau)} \right)\\
     \nonumber
      & \overset{\eqref{eq:max_smll}}{\le}      E_x \left ( e^{R_\lambda(\sigma_{\kappa^{-1} x})}\right)(1+c_1)\\
      \nonumber
      & \overset{\eqref{eq:Rlambdaxfar}}{\le} E_x \left ( e^{- c_2\gamma x^{\alpha} \sigma_{\kappa^{-1} x}}\right)(1+c_1). 
      \end{align}     
   Letting $\rho:= c_2  \gamma x^\alpha,~a:=\kappa^{-1} x$ and $b:=1-a$ in Proposition \ref{pr:momgen}-(1), we obtain 
      \begin{align}
      \nonumber 
          E_x \left ( e^{- c_2\gamma x^{\alpha} \sigma_{\kappa^{-1} x}}\right) &=  \frac{ \cosh (\sqrt{2\rho} (x-\frac 12 ))}{\cosh (\sqrt{2\rho}(\frac 12 - \frac{x}{\kappa}))}\le 2 \frac{e^{\sqrt{2\rho}(\frac 12 - x) }}{e^{\sqrt{2\rho} (\frac 12 - \frac{x}{\kappa})}}\\
\nonumber
        & =2 e^{-\sqrt{2\rho} (1-\kappa^{-1} )  x}=2 e^{-c_3 \gamma^{1/2} x^{\alpha/2 + 1} } \\
                \label{eq:scaling}
         &= 2 e^{-c_4 (x/r_2)^{\alpha/2+1}}.
                 \end{align}
Summarizing, we proved that for $x \in [r_2,\frac 12]$,
      \begin{equation}
      \label{eq:goodone} E_x \left ( e^{R_{\lambda} (\tau)}  \right) \le 2 (1+c_1)  e^{-c_4 (x/r_2) ^{\alpha/2 + 1} }.
      \end{equation} 
We are ready to complete the proof of the first bound in the lemma. We have 
              \begin{align*} E_{\mu} \left ( e^{R_\lambda(\tau)} \right) & \le 
        \int_{d(x,\partial D)\le r_2} E_x \left ( e ^{R_{\lambda}(\tau)} \right)  dx + \int_{d(x,\partial D)\ge r_2}  E_x \left ( e ^{R_{\lambda}(\tau)} \right)  dx \\ 
        & \overset {\eqref{eq:max_smll},\eqref{eq:goodone}}{\le} 2r_2 (1+c_1) + 4 (1+c_1)   \int_{r_2<x<\frac 12   }   e^{-c_4 (x/r_2) ^{\alpha/2 + 1} } 
        dx \\
         &\le  4 (1+c_1) r_2  \left ( 1  +  \int_1^\infty e^{-c_4 u^{\alpha/2+1}} du \right)=c_5 r.
         \end{align*}
         We turn the the second bound. The argument is similar. If $d(x,\partial D)  \ge  r_1$, we have 
         \begin{equation} 
         \label{eq:Maxfar} 
          E_x \left ( \int_0^{\sigma_{r_1} } e^{R_\lambda(t)} dt \right)  \overset{\eqref{eq:mainuppr} }{\le} E_x \left ( \int_0^{\sigma_{r_1}} e^{-\gamma r^\alpha t} dt\right) \le \frac {1}{\gamma r^\alpha}\overset{\eqref{eq:lambdasmll} }{\le}\frac{1}{\lambda}.
         \end{equation} 
Assume that  $d(x,\partial D)\le r_2$.  From Proposition \ref{pr:momgen}-(2)  we have 
$$ E_x \left ( e^{\lambda  \left (  \tau \wedge \sigma_{r_3}\right) } \right)=  \frac{ \sin( \sqrt{2\lambda} x)+\sin( \sqrt{2\lambda} (r_3-x))}{\sin(\sqrt{2\lambda} r_3)}\overset{ \eqref{eq:lambdasmll} }{\le} 2.$$
       Let $K\in\N$. From the strong Markov property we obtain 
    \begin{align*}
     E_x \left ( \int_0^{\tau} e^{R_{\lambda} (t)\wedge K} dt \right) &\le  E_x \left ( \int_0^{\tau\wedge \sigma_{r_3}} e^{\lambda t}  dt \right) + E_x \left ( e^{R_{\lambda} (\sigma_{r_3})}\ch_{\{\sigma_{r_3}<\tau\}}\right )  E_{r_3} \left ( \int_0^{\tau} e^{R_{\lambda} (t)\wedge K} dt \right).\\
     & \overset{\eqref{eq:allbest}}\le \frac {E_x \left ( e^{\lambda\left ( \tau \wedge \sigma_3\right)}\right) -1}{\lambda}+\frac{c_1}{1+c_1}  E_{r_3}  \left ( \int_0^{\tau} e^{R_{\lambda} (t)\wedge K} dt \right)\\
  & \le \frac{1}{\lambda} + \frac{c_1}{1+c_1}  E_{r_3}  \left ( \int_0^{\tau} e^{R_{\lambda} (t)\wedge K} dt \right).
    \end{align*}
But,
   \begin{align*}
E_{r_3}  \left ( \int_0^{\tau} e^{R_{\lambda} (t)\wedge K} dt \right)&\le  E_{r_3}\left (  \int_0^{\sigma_{r_1}} e^{R_{\lambda}(t)} dt \right)+E_{r_3} \left ( e^{R_{\lambda} (\sigma_{r_1}) }\right)  E_{r_1}\left (  \int_0^{\tau} e^{R_{\lambda}(t)\wedge K} dt \right)\\
    & \overset{\eqref{eq:Maxfar},\eqref{eq:mainuppr}}{\le} \frac 1\lambda  +  E_{r_1}\left (  \int_0^{\tau} e^{R_{\lambda}(t)\wedge K} dt \right).
   \end{align*}
 Combining the two upper bounds, we obtain 
   \begin{align} 
\label{eq:long_close}
     E_x \left ( \int_0^{\tau} e^{R_{\lambda} (t)\wedge K} dt \right) &\le \frac {1+2c_1}{1+c_1} \frac 1 \lambda + \frac{c_1}{1+c_1}  E_{r_1}\left (  \int_0^{\tau} e^{R_{\lambda}(t)\wedge K} dt \right)
     \end{align}
 Letting $x=r_1$, we obtain 
    $$ E_{r_1}\left (  \int_0^{\tau} e^{R_{\lambda}(t)\wedge K} dt \right) \le \frac{1+2c_1}{\lambda}, $$
    which in turn implies  
    $$ E_x \left ( \int_0^{\tau} e^{R_{\lambda} (t)\wedge K} dt \right)\le \frac {1}{\lambda} \left (  \frac{1+2c_1}{1+c_1}+ \frac{c_1}{1+c_1} \frac{1+2c_1}{1+c_1}\right) = \frac{1+2c_1}{\lambda}.$$
By letting $K\to\infty$, and using the monotone convergence theorem, we have proved that  
   \begin{equation}
   \label{eq:M2close} 
  E_x \left ( \int_0^\tau e^{R_\lambda(t)} dt\right) \le \frac{1+2c_1}{\lambda},
   \end{equation} 
 whenever  $d(x,\partial D)\le r_2$. \\
 
 Next we obtain an upper bound when $d(x,\partial D)\ge r_2$. We begin with an auxiliary bound.  Let $x \in [r_1,\frac 12]$. Then 
          \begin{align}
          \nonumber
         E_x \left( \int_0^{\tau} e^{R_{\lambda}(t)} dt \right)&= E_x \left ( \int_0^{\sigma_{r_1} } e^{ R_{\lambda} (t)} dt\right)  +E_x \left ( e^{R_{\lambda} (\sigma_{r_1})}\right)  E_{r_1} \left ( \int_0^{\tau} e^{R_{\lambda}(t)} dt\right) .\\
         \label{eq:M2unif}
            & \overset{\eqref{eq:Maxfar},\eqref{eq:mainuppr}}{\le}  \frac{ 1}{\lambda}+ \frac{1+2c_1}{\lambda}\le \frac{2(1+c_1)}{\lambda} .\end{align}
          We now obtain the main bound. Assume that  $x \in [r_2,\frac 12]$. We obtain  
           \begin{align*}
            E_x \left( \int_0^{\tau} e^{R_{\lambda}(t)} dt \right)&= E_x \left ( \int_0^{\sigma_{\kappa^{-1} x } } e^{ R_{\lambda} (t)} dt\right)  +E_x \left ( e^{R_{\lambda} (\sigma_{\kappa^{-1} x})}\right)  E_{\kappa^{-1} x} \left ( \int_0^{\tau} e^{R_{\lambda}(t)} dt\right) .\\
         &  \overset{\eqref{eq:Rlambdaxfar},\eqref{eq:scaling}}{  \le} E_x \left ( \int_0^{\sigma_{\kappa^{-1} x}}  e^{-c_2\gamma x^{\alpha} s} ds \right)  +  2 e^{-c_4 (x/r_2(\gamma)) ^{\alpha/2 + 1} }E_{\kappa^{-1} x} \left( \int_0^{\tau} e^{R_{\lambda}(t)} dt \right) \\
          & \overset{\eqref{eq:M2unif}}{\le} \frac{1}{c_2 \gamma  x^{\alpha}}+  \frac{4(1+c_1)e^{-c_4 (x/r_2(\gamma)) ^{\alpha/2 + 1}}}{\lambda}.         
         \end{align*}
          Therefore 
           \begin{align*} 
  \lambda  \int_{d(x,\partial D)\ge r_2}  E_x \left( \int_0^{\tau} e^{R_{\lambda}(t)} dt \right) 
           dx &\le 2  \theta h  \int_{r_2\le x<\frac 12} \frac{1}{c_2 x^\alpha} dx+ 4 (1+c_1)r_2    \int_1^\infty e^{-c_4 u^{\alpha/2+1}} du.\\
           \overset{\tiny{\mbox{Lemma }}\ref{lem:hasit}-(2)}{\le} c_6 (\theta +1)r.  
           \end{align*}
          Along with  \eqref{eq:M2close} we obtain 
          $$ \lambda  E_{\mu} \left ( \int_0^\tau e^{R_\lambda(t)} dt \right) \le 2 (1+2c_1) r_2 + c_6 (\theta+1) r =(\theta+1) c_7 r.$$
           Let $\theta_0:=  \min  (\frac{\theta_1}{2},1).$  If 
            $\theta > \theta_0$, then  $(1+\theta )< 2 \theta $. When $ \theta \le \theta_0$ we have 
          $$ \lambda  E_{\mu} \left ( \int_0^\tau e^{R_\lambda(t)} dt \right) \le   \frac{\lambda} {\lambda(\theta_0)} \lambda(\theta_0)E_{\mu} \left ( \int_0^\tau e^{R_{\lambda(\theta_0)}(t)} dt \right)\le 
          \theta \frac{(1+\theta_0)}{\theta_0} c_7 r ,$$ 
          and the result follows. 
\end{proof} 
For real  $\theta$, let $v_{\theta}:= 1+  \theta_-$, where $\theta_- := \max(-\theta, 0)$.
\begin{lemma}
\label{lem:lowerbd}

There exist positive constants $C_3,C_4,\gamma_1$ depending only on $V$, such that for $\gamma >\gamma_1$ and $\theta \in \R$
   $$ E_\mu \left (  e^{R_{\lambda} (\tau)}\right)  \ge C_3 \frac{ r(\gamma)}{\sqrt{v_{\theta}}} ,$$
  and 
  $$ |\lambda | E_\mu \left ( \int_0^{\tau} e^{R_{\lambda}(t)} dt \right)\ge  C_4 \frac{ |\theta|}{v_{\theta}}  r(\gamma).$$
\end{lemma} 
\begin{proof} 
We begin with some preparation. Let $\eta$ be a positive constant satisfying  $V(x) \le \eta x^{\alpha}$ for all $x\in D$. By definition of $\lambda$ and Lemma \ref{lem:hasit}-(1),  $\lambda= \theta \gamma h \ge  -\theta_-\gamma r^{\alpha}$. Now let $x \in [0,r]$, and let $0\le s\le  t \le \tau_{2x}$. Then $P_x$ a.s. we have $V(X(s)) \le \eta X(s)^{\alpha}\le \eta 2^{\alpha} r^{\alpha}$. Therefore 
     \begin{equation*}
     \label{eq:Rlmlrg} 
     R_\lambda (t) \ge - \theta_- \gamma r^{\alpha}t  - \gamma c_1 r^{\alpha}  t \ge - c_1  v_\theta \gamma r^{\alpha} t.
     \end{equation*} 
     Then,  
    $$ E_x \left ( e^{R_{\lambda}(\tau)}\right) \ge E_x \left ( e^{R_{\lambda}(\tau)}\ch_{\{\tau< \tau_{2x}\}}\right) \ge E_x \left ( e^{-c_1v_{\theta}  \gamma r^\alpha \tau}\ch_{\{\tau< \tau_{2x}\}} \right) .$$ 
Letting $\rho:=  c_1 v_{\theta} \gamma r^\alpha$, $a:=0,b=2x$ and $y:=x$ in Proposition \ref{pr:momgen}-(1) we obtain 
$$E_x \left ( e^{-c_1 \gamma r^\alpha \tau}\ch_{\{\tau< \tau_{2x}\}} \right) = \frac{ \sinh (\sqrt{2\rho} x ) }{\sinh(\sqrt{2\rho}2x )}=\frac{1}{2\cosh(\sqrt{2\rho}x)}\ge \frac{1}{2 e^{\sqrt{2\rho} x}}.$$
 Since $\sqrt{2\rho}   = \sqrt{2 c_1v_{\theta} }  \gamma^{1/2}  r^{\alpha/2} =\sqrt{2 c_1v_{\theta} }r^{-1} $,   we conclude that 
\begin{align*}
  E_\mu \left ( e^{R_{\lambda} (\tau)}\right)   &\ge \int_{0<x<r} E_x \left ( e^{R_{\lambda}(\tau)}\ch_{\{\tau<\tau_{2x}\}}\right)dx \\
   & \ge \frac 12 \int_{0<x<r} e^{-\sqrt{2 c_1v_{\theta}}r^{-1} x} dx \\
    & = \frac r2  \int_0^1 e^{-\sqrt{2 c_1v_{\theta}}y} dy\\
     &  = \frac  {1-e^{-\sqrt{2c_1v_{\theta}}  }}{2\sqrt{2c_1 v_{\theta}}} r \ge \frac{1-e^{-\sqrt{2c_1}} }{2\sqrt{2c_1 v_{\theta}}}r = \frac{c_2 r }{\sqrt{v_{\theta}}}, \end{align*}
     and $c_2$ is a positive constant independent of $\theta$. This completes the proof of the first bound. \\
     
      We turn to the second  bound. Fix  $x \in [2r,\frac 13]$. We have 
 $$ E_x \left ( \int_0^{\tau} e^{R_\lambda (t)} dt\right)  \ge E_x \left ( \int_0^{\tau_{0.5 x} \wedge \tau_{1.5x} }e^{R_\lambda (t)} dt\right ).$$
  Let $0\le s\le t  \le \tau_{0.5x} \wedge \tau_{1.5 x}$. Then  $ V(X(s)) \le \eta (1.5 x)^{\alpha}$, and since  $\lambda = \theta \gamma h \ge -\gamma \theta_- r^{\alpha} \ge -\gamma \theta_- (0.5x)^\alpha$, we conclude with 
  $$ R_\lambda (t) \ge - \gamma (\theta_-0.5^{\alpha} +\eta 1.5^{\alpha}) x^{\alpha} t =-c_3 v_\theta\gamma  x^{\alpha} t.$$ 
  We then have 
    $$   E_x \left ( \int_0^{\tau_{0.5 x} \wedge \tau_{1.5x} }e^{R_{\lambda} (t)} dt\right )\ge  \frac{1- E_x \left( e^{-c_3 v_{\theta} \gamma x^{\alpha} \tau_{0.5x}\wedge \tau_{1.5x}}\right)}{c_3v_{\theta}\gamma x^{\alpha}} .$$ 
    From Proposition \ref{pr:momgen}-(1) with $\rho := c_3v_{\theta} \gamma x^{\alpha}$,  $y:=x,~a:=0.5x$ and $b:=1.5x$ to obtain 
$$E_x(  e^{-\rho \left( \tau_{0.5x}\wedge \tau_{1.5x}\right)})=\frac{ 1}{\cosh(\sqrt{2\rho}0.5x)}.$$
 Observe that  $\sqrt{2\rho}   0.5 x \ge 0.5  \sqrt{2c_3} \gamma^{1/2} r^{\alpha/2+1}= 0.5 \sqrt{2 c_3}$. Therefore 
$$  E_x \left ( \int_0^{\tau_{0.5 x} \wedge \tau_{1.5x} }e^{R_\lambda(t)} dt\right )\ge \frac{1- \frac{1}{\cosh (\sqrt{2\rho} 0.5x)}}{\rho }\ge \frac{c_4} {v_{\theta} \gamma x^{\alpha} }, $$
where the positive constant $c_4$ is independent of $\theta$. Integrating this inequality we obtain 
  $$ E_{\mu} \left ( \int_0^{\tau} e^{R_\lambda (t)} dt\right) \ge \frac{c_4}{\gamma v_{\theta}} \int_{2r<x<\frac 13}  \frac{1}{x^\alpha} dx,$$
The result now follows from  Lemma \ref{lem:hasit}-(2). 
\end{proof} 
\section{Proof of Theorem \ref{th:asymp}}
In this section we use the results of the preceding sections to prove Theorem \ref{th:asymp}. 
\begin{proof}[Proof of Theorem \ref{th:asymp}]
\label{sec:proof} 
 Let $\lambda=\lambda(\theta,\gamma,\alpha)$  be the function defined in \eqref{eq:lambda}. 
 We first obtain a lower bound on $\lambda_0(\gamma)$. 
 It follows from   Proposition \ref{pr:Rmomgen}-(4) and Lemmas \ref{lem:upperbd} and \ref{lem:lowerbd}, that 
  $$ E_\mu^{\gamma} \left ( e^{\lambda \tau} \right) \le \frac{1}{1 - \frac{C_2 \theta }{C_3} },$$
   for $\theta \in (0,\theta_1)$ and all $\gamma$ sufficiently large. In particular letting $\theta =  \frac 12 \min (\frac{C_3}{C_2},\theta_1)$, we obtain that $E_x^{\gamma} \left (e^{\lambda \tau} \right)$ is finite for some $x \in D$. We conclude from \eqref{eq:momgen_rep} that  $\lambda \le \lambda_0(\gamma)$, completing the proof of the lower bound on $\lambda_0(\gamma)$. \\
   
   We turn to the upper bound. In light of \eqref{eq:momgen_rep}, in order to show that $\lambda \ge \lambda_0(\gamma)$, it is sufficient to show that $E_x^{\gamma} \left (e^{\lambda \tau} \right)=\infty$ for some $x\in D$. However, by Proposition \ref{pr:Rmomgen}-(1) this  condition holds if  $E_{\mu}^\gamma ( e^{\lambda \tau}) =\infty$. This is what we will prove. We split the discussion according to the value of $\alpha$. \\
   
    Assume first that $\alpha\le 1$.  From Lemmas \ref{lem:lowerbd} and \ref{lem:upperbd} we conclude that there exist positive constants depending only on $V$ such that for every $\theta >0$, there exists  $\gamma_1:=\gamma_1(\alpha,\theta)\in (0,\infty)$ and 
  $$ |\lambda| E_\mu  \left ( \int_0^\tau e^{R_\lambda(t)} dt \right) \ge C_4 \theta r,\mbox{ and } E_\mu \left ( e^{R_\lambda(\tau)}\right) \le C_1 r,$$ 
  provided $\gamma > \gamma_1$. Furthermore, $\theta \to \gamma_1(\alpha,\theta)$ is nondecreasing, hence the above inequalities hold  for all  $0<\theta< \frac{2C_1}{C_4}$, if $\gamma \ge  \gamma_1(\alpha, \frac{2C_1}{C_4})$. But then,  Proposition \ref{pr:Rmomgen}-(4)  gives 
  $$ \liminf_{\theta \nearrow \frac{C_1}{C_4} } E_\mu^\gamma \left ( e^{\lambda \tau} \right) \ge \lim_{\theta \nearrow \frac{C_1}{C_4}} \frac{1}{1-\frac{C_4 \theta}{C_1}}=\infty.$$ 
  In particular, for $\theta : = \frac{C_1}{C_4}$, we have $E_{\mu}^\gamma \left ( e^{\lambda \tau } \right) = \infty$. \\

 Finally, assume that  $\alpha>1$. Note that  the upper bounds of Lemma \ref{lem:upperbd} may not hold for all $\theta$, so the argument in the last paragraph may not work. Recalling from \eqref{eq:lambda} that  $\lambda =\theta \gamma^{\frac{2}{\alpha+2}}$, it follows from Proposition \ref{pr:crudebd} that there exists a constant $\theta_0\in (0,\infty)$ such that for $\theta>\theta_0$,  we have  $E_{\mu}^\gamma \left ( e^{\lambda \tau} \right) =\infty$.  
    \end{proof} 
\bibliographystyle{amsalpha}
\bibliography{instj_bdd}
\end{document}